\documentclass[reqno]{amsart}

\usepackage[utf8]{inputenc}
\usepackage{amsmath}
\usepackage{amsthm}
\usepackage{amsfonts}
\usepackage{amssymb}
\usepackage{caption}
\usepackage{hyperref} 
\usepackage[english]{cleveref} 
\usepackage{dsfont}
\usepackage{emptypage}
\usepackage[shortlabels]{enumitem}
\usepackage{extpfeil}
\usepackage{faktor}
\usepackage{fancyhdr}
\usepackage{float}
\usepackage{graphicx}
\usepackage{indentfirst}
\usepackage{mathdots}
\usepackage{mathtools}
\usepackage{mathrsfs}
\usepackage{multirow}
\usepackage{rotating}
\usepackage{stackrel}
\usepackage{stmaryrd}
\usepackage{subcaption}
\usepackage{tikz}
\usepackage{verbatim}

\usepackage[english]{babel}
\usepackage{csquotes}
\usepackage[backend=bibtex, style=alphabetic, doi=false, isbn=false, url=false, maxnames=5]{biblatex}
\usetikzlibrary{cd}

\newtheorem{thm}{Theorem}[section]
\newtheorem{prop}[thm]{Proposition}

\newtheorem{lemma}[thm]{Lemma}
\theoremstyle{definition}
\newtheorem{definition}[thm]{Definition}

\newtheorem*{question}{Question}
\theoremstyle{remark}

\DeclareMathOperator{\codim}{codim}

\DeclareMathOperator{\Rad}{Rad}
\DeclareMathOperator{\rk}{rk}
\DeclareMathOperator{\LG}{LG}
\DeclareMathOperator{\Gr}{Gr}

\newcommand{\oomega}{\overline{\omega}}
\def\dd{\textrm{d}}
\def\N{\mathbb{N}}
\def\Z{\mathbb{Z}}
\def\Q{\mathbb{Q}}
\def\R{\mathcal{R}}
\def\C{\mathbb{C}}
\def\PP{\mathbb{P}}
\def\A{\mathcal{A}}
\def\S{\mathcal{S}}
\newcommand{\defeq}{\stackrel{\textnormal{def}}{=}}

\makeatletter
\newcommand*{\bigcdot}{%
  {\mathbin{\mathpalette\bigcdot@{}}}%
}
\newcommand*{\bigcdot@scalefactor}{.75}
\newcommand*{\bigcdot@widthfactor}{1.4}
\newcommand*{\bigcdot@}[2]{%
  \sbox0{$#1\vcenter{}$}
  \sbox2{$#1\cdot\m@th$}%
  \hbox to \bigcdot@widthfactor\wd2{%
    \hfil
    \raise\ht0\hbox{%
      \scalebox{\bigcdot@scalefactor}{%
        \lower\ht0\hbox{$#1\bullet\m@th$}%
      }%
    }%
    \hfil
  }%
}
\makeatother

\title{Two Examples of Toric Arrangements}
\author{Roberto Pagaria}
\address{Roberto Pagaria}
\address{\textup{Scuola Normale Superiore\\ Piazza dei Cavalieri 7, 56126 Pisa\\ Italia}}
\email{roberto.pagaria@sns.it}

\begin{document}

\begin{abstract}
We show that the integral cohomology algebra of the complement of a toric arrangement is not determined by the poset of layers.
Moreover, the rational cohomology algebra is not determined by the arithmetic matroid (however it is determined by the poset of layers).
\end{abstract}

\maketitle

\section*{Introduction}
A toric arrangement is a finite collection of hypertori lying in an algebraic torus.
The complement of a toric arrangement is an open set: we are interested in studying its cohomology ring.
The combinatorics of an arrangement is encoded in its poset of layers, i.e.\ the poset of connected components of the intersections of some hypertori.
In the case of hyperplane arrangements, the associated poset we consider is the poset of intersections, which turns out to be a geometric lattice.
The combinatorial data of a central hyperplane arrangement can be stored equivalently in an another structure: a matroid.
There are several equivalent definitions of matroid, we point to \cite{Oxley} for a general reference.
Matroids can be generalized to the toric case in different ways:
arithmetic matroids (see \cite{DAdderio-Moci} and \cite{Branden-Moci}), matroids over rings (see \cite{Fink-Moci}) and $G$-semimatroids (see \cite{Delucchi-Riedel}).
All these combinatorial data permit us to define the arithmetic Tutte polynomial (introduced first in \cite{Moci}), that is the analogous of the Tutte polynomial of a hyperplane arrangement.

The study of interplay between the cohomology of the complement of a toric arrangement and its combinatorics started in \cite{Looijenga}, where the Betti numbers of the complement were computed: it follows that the Poincaré polynomial is a specialization of the arithmetic Tutte polynomial (see \cite{Moci}).
The cohomology algebra of the complement of a hyperplane arrangement is the cohomology of the algebraic de Rham complex (see \cite{Grothendieck}) that  was combinatorially described in \cite{OS80}.
An analogous approach in the toric case started with \cite{DP05} (see also \cite{DPbook}), where a description of the associated graded cohomology algebra with complex coefficients $\Gr H(M(\A),\C)$ was given.
The filtration used in \cite{DP05} coincides with the Leray filtration for the inclusion of the complement in the ambient torus.
The graded algebra with rational coefficients $\Gr H(M(\A),\Q)$ can be obtained from the Leray spectral sequence as shown in \cite{Bibby} and \cite{DupontHypersurfaces}.
The graded algebra with integer coefficients $\Gr H(M(\A),\Z)$ was studied in \cite{Callegaro-Delucchi} and from the combinatorial point of view, in \cite{Pagaria}.
Recently, presentations of the cohomology algebras $H(M(\A),\Q)$ and $H(M(\A),\Z)$ in the spirit of \cite{OS80} was obtained in \cite{CDDMP}, generalizing \cite{DP05}.
The description of the rational cohomology ring $H(M(\A),\Q)$ depends only on the poset of layers.

\medskip

In \Cref{sect:first_example} we show that the integral cohomology algebra $H(M(\A),\Z)$ of the complement of a (central) toric arrangement is not combinatorial, i.e.\ it does not depend only on the poset of layers (\Cref{thm:main_integer}).
This example gives a negative answer to Question 7.3.1 of \cite{Callegaro-Delucchi}.

In \cref{sect:second_example}, we show that arithmetic matroids and matroids over $\Z$ contain less information than the poset of layers.
Indeed, we build two central toric arrangements with the same arithmetic matroid, the same matroid over $\Z$, but with non-isomorphic posets of layers (\Cref{thm:main_combin}) and non-isomorphic cohomology algebra with rational coefficients.
As consequence, there cannot exist a ``cryptomorphism'' between arithmetic matroids (respectively, matroids over $\Z$) and any class of posets such that -- in the representable cases -- the poset associated with the matroid coincides with the poset of layers of any representation.

The following question about central toric arrangements remains open.
\begin{question}
Does the integral cohomology algebra of the complement of a central toric arrangement determine the toric arrangement?
\end{question}

\section{Definitions}
Let $N=(n_{i,j})$ be a matrix with integer coefficients of size $r \times n$.

\begin{definition}
The \textit{central toric arrangement} $\A$ defined by $N$ is the collection of $n$ hypertori $H_1, \dots, H_n$ in $T=(\C^*)^r$, where
\[ H_j \defeq \left\{ (x_1, \dots x_r) \in (\C^*)^r \mid x_1^{n_{1,j}} x_2^{n_{2,j}} \cdots x_r^{n_{r,j}}=1 \right\}. \]
\end{definition}
We are interested in studying the topological invariants of the \textit{complement} \[M(\A) \defeq T \setminus \bigcup_{H \in \A} H\]
of a toric arrangement $\A$.

We work only with central toric arrangements.
The centrality of the arrangements will be understood.

\begin{definition}
Let $\A=\{H_1, \dots, H_n\}$ be a toric arrangement.
A \textit{layer} $W$ of $\A$ is a connected component of the intersection of some hypertori in $\A$.
The \textit{poset of layers} $\S(\A)$ of the toric arrangement $\A$ is the partially ordered set whose elements are all the layers of $\A$ ordered by reverse inclusions.
The poset of layers is ranked by the codimension in $T$
\[ \rk W = \codim_{T} W.\]
\end{definition}


For each pair of layers $(W_1,W_2)$ there exists a unique \textit{meet} $W_1 \wedge W_2$ of the two layers, i.e.\ the minimal layers containing both $W_1$ and $W_2$.
However, a \textit{join} 
of two elements 
may not be unique (e.g. if the intersection of the two tori is not connected).

In the final part of this paper we will need the notions of arithmetic matroids and of matroids over $\Z$ (also known as $\Z$-matroids).
A \textit{matroid} is a pair $(E,\rk)$ where $E$ is a finite set and $\rk : \mathcal{P}(E) \rightarrow \N$ is a \textit{rank function} that satisfies:
\begin{enumerate}
\item $\rk(S) \leq |S|$ for all $S\subseteq E$,
\item $\rk(S) \leq \rk(T)$ for all $S \subseteq T \subseteq E$,
\item $\rk(S\cap T)+ \rk(S\cup T) \leq \rk(S)+\rk(T)$ for all $S,T \subseteq E$.
\end{enumerate}

An \textit{arithmetic matroid} is a matroid $(E,\rk)$ with a \textit{multiplicity function} $m:\mathcal{P}(E) \rightarrow \N\setminus \{0\}$ that satisfies five properties (listed for instance in \cite{DAdderio-Moci} or \cite{Branden-Moci} for the definition).

A \textit{matroid over} $\Z$ is a ground set $E$ together with a $\Z$-module $M(S)$ for each subset $S$ of $E$ such that these modules satisfy some specific relations (see \cite{Fink-Moci} for the definition).

We will give the definition of arithmetic matroids and matroids over $\Z$ only in the representable cases, since we deal only with the representable ones.

Let $N\in M(r,n;\Z)$ be an integer matrix.
We denote with $N[S]$, for $S\subseteq [n]$, the sub-matrix of $N$ made up with the columns indexed with $S$.
The \textit{representable matroid over} $\Z$ described by $N$ is the function $\mathcal{M}$ that associate to every subset $S\subseteq [n]$ the module $\Z^n/ \langle N[S] \rangle$, where $\langle N[S] \rangle$ is the sub-module generated by the columns of $N[S]$.
The \textit{representable arithmetic matroid} described by $N$ is $([n],\rk,m)$, where $\rk(S)= \operatorname{rank} \langle N[S] \rangle$ and  $m(S)= |\operatorname{tor}(\Z^n/ \langle N[S] \rangle )|$ is the cardinality of the torsions of the quotient.

In general matroids over $\Z$ contain more information than the corresponding arithmetic matroid.
However, representable matroids with $m(E)=m(\emptyset)=1$ determine the isomorphisms classes of the torsion subgroups $\operatorname{tor}(\Z^n/ \langle N[S] \rangle)$ for all $S\subset E$.
In this case the arithmetic matroid determines the matroid over $\Z$.

\begin{definition}
A toric arrangement is \textit{totally unimodular} if all intersections of some hypertori are connected.
\end{definition}

We recall the presentation of the cohomology ring $M(\A)$, 
it will be useful later.
For each hypertorus $H_i \in \A$ with equation $\underline{t}^{\underline{n}^i}=1$ we define the differential forms $\omega_i= \frac{1}{2 \pi \sqrt{-1}}\dd \log (1- \underline{t}^{\underline{n}^i})$ and $\psi_i= \frac{1}{2 \pi \sqrt{-1}} \dd \log (\underline{t}^{\underline{n}^i})$.
For the sake of notation we assume a total order on $\A$ and a defining matrix $N=(\underline{n}^i)_{i=1,\dots,n} \in M(r,n;\Z)$.
\begin{thm}\label{lemma3.2}
Let $\A=\{H_1, \dots, H_n\}$ be a totally unimodular toric arrangement.
The cohomology $H^\bigcdot(M(\A);\Z)$ is generated in degree one by the differential forms $\omega_i$, $\psi_i$.
The relations are:
\begin{itemize}
\item $\omega_i \psi_i=0$ for every $H_i\in \A$,
\item for every $I\subseteq \A $ and $c_i \in \{\pm 1\}$ such that $\sum_{H_i\in I} c_i \underline{n}^i=0$, a relation
\[\sum_{H_i \in I} c_i \psi_i =0,\]
\item for every $I\subseteq \A $ and $c_i \in \{\pm 1\}$ such that $\sum_{H_i\in I} c_i \underline{n}^i=0$, a relation
\[\prod_{H_i \in I} (\omega_i-\omega_{i-1}+ c_i \psi_{i-1} ) = 0.\]
\end{itemize}
\end{thm}

The previous theorem follows from \cite[Theorem 5.2]{DP05} and \cite[Lemma 3.2]{CDDMP}.
The following theorem is \cite[Theorem 6.13]{CDDMP}.

For each ordered subset $S\subseteq \A$ and each connected component $W$ of $\cap_{H \in S} H$ such that $|S|=\rk W$, there exist a differential form $\oomega_{W,S}$ of degree $|S|$ defined in \cite{CDDMP}.
If $S=\{H_i\}$ then the form $\oomega_{H_i,S}$ coincides with $2\omega_i -\psi_i$.

\begin{thm}\label{thm6.13}
Let $\A$ be a central toric arrangement.
The rational cohomology algebra $H^*(M(\A),\mathbb Q)$ is generated by $\oomega_{W,S}$, $\psi_i$ with relations:
\begin{itemize}
\item $\oomega_{W,S} \psi_i=0$ for every $H_i\in S \subseteq \A$,
\item $\oomega_{W,S} \oomega_{V,T}=0$ if $\rk W+\rk V> \rk W\cap V$ and if $\rk W+\rk V= \rk W\cap V$ a relation
\[\oomega_{W,S} \oomega_{V,T}=(-1)^{l(S,T)} \sum_{U \textnormal{ c.c. of } W \cap V} \oomega_{U, S\cup T},\]
where $l(S,T)$ is the parity of the permutation that reorder $S\cup T$.
\item for every $I\subseteq \A $ and $k_i \in \Z$ such that $\sum_{H_i\in I} k_i \underline{n}^i=0$, a relation
\[\sum_{H_i \in I} k_i \psi_i =0,\]
\item for every $I\subseteq \A $, connected component $L$ of $\cap_{H_i \in I} H_i$ and $k_i \in \Z$ such that $\sum_{H_i\in I} k_i \underline{n}^i=0$, a relation
\[\sum_{H_i \in I} \sum_{S,T} (-1)^{|S_{\leq i}|+l(S,T)} c_T \frac{m(S)}{m(S \cup T)} \oomega_{W,S} \prod_{H_i \in T} \psi_i = 0,\]
where $S \sqcup T= I \setminus \{ H_i \}$, $|T|$ is even and $W$ is the connected component of $\cap_{H_i \in S} H_i$ containing $L$.
The integer $m(J)$ is the number of connected components of $\cap_{H_i \in J} H_i$, $c_T$ is the sign $\prod_{H_i \in T} \operatorname{sgn} k_i$ and $|S_{\leq i}|$ is the number of elements of $S$ that precede $H_i$ in the chosen total order of $\A$.
\end{itemize}		
\end{thm}

\medskip

Let $A= \bigoplus_{n \in \N} A^n$ be a graded-commutative $R$-algebra and consider for each $\alpha \in A^1$ the left multiplication $\delta_\alpha^i: A^i \rightarrow A^{i+1}$. The pair $(A;\delta_\alpha)$ is a complex for each $\alpha \in A^1$.
\begin{definition}
The $k^{\textnormal{th}}$ \textit{resonance variety} of $A$ is
\[ \R^k(A) \defeq \{ \alpha \in A^1 \mid H^k(A, \delta_\alpha) \neq 0\}. \]
The $k^{\textnormal{th}}$ resonance varieties (with coefficients in the domain $R$) for a toric arrangement $\A$ is
\[ \R^k(\A;R) \defeq \R^k ( H^\bigcdot (M(\A);R)).\]
\end{definition}
We will use only the first resonance variety $\R^1(\A,R)$ of a toric arrangement $\A$, where $R$ is the ring $\Z$ or $\Q$.

\section{First example}\label{sect:first_example}
The example that we will expose in this section was already appeared in \cite[Example 7.1]{Pagaria} as a generalization of  \cite[Example 7.3.2]{Callegaro-Delucchi}.
Without a complete description of the cohomology ring, it was not possible to complete the calculation.
Now, using \Cref{thm6.13} and \Cref{lemma3.2} we can perform that computation.

In this section we set $T=(\C^*)^2$.
Consider the arrangements $\A$ and $\A_n^a$ in $T$ defined respectively by the matrices
\[
N = \begin{pmatrix}
1 & 0 & 1 \\ 
0 & 1 & 1
\end{pmatrix}
\textnormal{ and }
N_n^a = \begin{pmatrix}
1 & a & a+1 \\ 
0 & n & n
\end{pmatrix},
\]
where $n$ is a positive integer and $a,a+1$ are relatively prime to $n$.

We use \cite[subsection 2.4.3]{Looijenga} or \cite[Corollary 5.12]{Moci} to calculate the Poincaré polynomials of these arrangements.
The Poincaré polynomial of $M(\A)$ is $1+5t+6t^2$ and that of $M(\A_n^a)$ is $1+5t+(2n+4) t^2$.
The Tutte polynomial of the arithmetic matroid associated with $N$ is $x^2+x+y$, the one associated with $N_n^a$ is $x^2+x+ny+2n-2$.

\begin{thm}\label{thm:main_integer}
Let $n>5$ be a natural number relatively prime to $6$, the arrangements $\A_n^1$ and $\A_n^2$ have isomorphic posets of layers but non isomorphic cohomology algebras with integer coefficients.
\end{thm}

From \Cref{thm6.13} the two arrangements $\A_n^1$ and $\A_n^2$ have isomorphic cohomology algebras with rational coefficients.
We need a couple of lemmas to prove \Cref{thm:main_integer}.

\begin{lemma}
Let $A$ be a graded-commutative algebra over $\Q$.
The first resonance variety $\R^1(A)$ is a union (possibly infinite) of planes in $A^1$.
\end{lemma}

\begin{proof}
If $\alpha \in \R^1(A)$, then there exists $\beta \in A^1 \setminus \alpha \Q$ such that $\alpha \beta=0$.
Thus, the plane generated by $\alpha$ and $\beta$ is contained in $\R^1(A)$.
We obtain the desired result from the arbitrariness of $\alpha \in \R^1(A)$.
\end{proof}

We use coordinates $t_1,t_2$ on $T$ and we apply \Cref{lemma3.2}.
The cohomology ring of $M(\A)$ is generated by the closed forms 
\begin{align*}
\omega_1 &= \frac{1}{2 \pi \sqrt{-1}} \dd \log (1-t_1),\\
\omega_2 &= \frac{1}{2 \pi \sqrt{-1}} \dd \log (1-t_2),\\
\omega_3 &= \frac{1}{2 \pi \sqrt{-1}} \dd \log (1-t_1t_2),
\end{align*}
associated with the hypertori $H_1,H_2, H_3$ respectively, together with the forms $\psi_1= \frac{1}{2 \pi \sqrt{-1}} \dd \log (t_1)$ and $\psi_2= \frac{1}{2 \pi \sqrt{-1}} \dd \log (t_2)$ ($\psi_3$ is equal to $\psi_1+\psi_2$).
The relations are:
\begin{equation}\label{eq:rel_in cohomology}
\begin{gathered}
\omega_1\omega_2- \omega_1 \omega_3+ \omega_2\omega_3 - \omega_3\psi_1 = 0, \\
\omega_1 \psi_1=0, \\
\omega_2 \psi_2=0, \\
\omega_3\psi_1+ \omega_3\psi_2=0.
\end{gathered}
\end{equation}

\begin{lemma}\label{lemma:R1(A)}
The first resonance variety $\R^1(\A;\Q)$ of the complement of $\A$ is the union of the following five planes of $H^1(M(\A);\Q)$;
\begin{align*}
P_1 &=\langle \omega_1, \psi_1 \rangle, \\
P_2 &=\langle \omega_2, \psi_2 \rangle, \\
P_3 &=\langle \omega_3, \psi_1+\psi_2 \rangle, \\
P_4 &=\langle \omega_1-\omega_3, \omega_1-\omega_2 - \psi_1 \rangle, \\
P_5 &=\langle \omega_2-\omega_3, \omega_1-\omega_2 + \psi_2 \rangle.
\end{align*}
\end{lemma}

\begin{proof}
The multiplication map $f: H^1(M(\A)) \otimes H^1(M(\A)) \rightarrow H^2(M(\A))$ is surjective and factors through $\bigwedge^2 H^1(M(\A))$.
The kernel of 
\begin{align*}
\tilde{f}\colon  \bigwedge^2 H^1(M(\A)) & \longrightarrow  H^2(M(\A)) \\
 \alpha \wedge \beta & \longmapsto  \alpha\beta
\end{align*}
has dimension $4=\binom{5}{2}-6$, hence $L \defeq \PP(\ker \tilde{f}) \simeq \PP^3$ is a linear subspace of $\PP(\bigwedge^2 H^1(M(\A))) \simeq \PP^9$.

An element $\alpha \in H^1(M(\A))$ belongs to the first resonance varieties if and only if there exists $\beta \in H^1(M(\A))$ such that $\alpha \beta=0$ in $ H^2(M(\A))$ and $\beta \not \in \C \alpha$.
This implies that $\alpha \wedge \beta$ is in $\ker \tilde{f}$ and so $[\alpha \wedge \beta] $ is in the linear subspace $L$.
Viceversa if $[\gamma]$ belongs to $L$ and is a decomposable tensor (i.e.\ belongs to $\Gr(2,H^1(M(\A)))$) then $[\gamma]= [\alpha \wedge \beta]$ and the plane $\langle \alpha, \beta \rangle$ is contained in the first resonance variety.

Now we prove that the intersection $L \cap \Gr(2,H^1(M(\A)))$ is the disjoint union of five points.
The relations in eq.~\eqref{eq:rel_in cohomology} implies the following factorized equations
\begin{align*}
(\omega_1-\omega_3)( \omega_1-\omega_2 - \psi_1) & = 0, \\
(\omega_2-\omega_3)( \omega_1-\omega_2 + \psi_2) & = 0.
\end{align*}
These equations ensure that the five different points $[P_i]$, $i=1,\dots, 5$ lie in this intersection.
The dimension of the Grassmannian $\Gr(k,V)$ is $k(\dim V -k)$, which in our case is equal to $6$.
Moreover, when $k=2$ its degree coincides with the Catalan number $C_{\dim V-2} $.
The formula for the degree of the Pl\" ucker embedding of the Grassmannian is due to Schubert in 1886, we refer to \cite{GW11}.
Hence $\Gr(2,H^1(M(\A)))$ has degree $5$ and every $\PP^3 \subset \PP^9$ intersects $\Gr(2,5)$ scheme-theoretically in five points (this is the general case) or in a sub-variety of positive dimension.

We exclude the latter case by explicit computation.
Fix the Pl\" ucker coordinates $[x_{ij}]_{1\leq i < j \leq 5}$ of $\PP^9$, where $\{ \omega_1, \omega_2, \omega_3, \psi_1, \psi_2\}$ is the chosen basis of $H^1(M(\A))$.
The coordinates of the five planes -- in lexicographical order $[x_{1,2},x_{1,3},x_{1,4}, \dots ,x_{4,5}]$ -- are:
\begin{align*}
P_1 &= [0,0,1,0,0,0,0,0,0,0], \\
P_2 &= [0,0,0,0,0,0,1,0,0,0], \\
P_3 &= [0,0,0,0,0,0,0,1,1,0], \\
P_4 &= [1,-1,1,0,1,0,0,-1,0,0], \\
P_5 &= [1,-1,0,0,1,0,-1,0,1,0].
\end{align*}
Thus the linear subspace $L$ has equation given by the ideal 
\[I \defeq (x_{15}, x_{24}, x_{45}, x_{12}+x_{13}, x_{13}+x_{23}, x_{13}-x_{34}+x_{35}).\]
The equation of the Grassmannian are given by the Pfaffian of principal minors of size four of a skew-symmetric matrix. Thus the defining ideal is
\begin{align*}
J \defeq (x_{12}x_{34}-x_{13}x_{24}+x_{14}x_{23}, x_{12}x_{35}-x_{13}x_{25}+x_{15}x_{23},\\
x_{12}x_{45}-x_{14}x_{25}+x_{15}x_{24}, x_{13}x_{45}-x_{14}x_{35}+x_{15}x_{34}, \\
x_{23}x_{45}-x_{24}x_{35}+x_{25}x_{34})
\end{align*}
and the sum of the two ideals is
\begin{align*}
I+J=(x_{15}, x_{24}, x_{45}, x_{14}x_{25}, x_{14}x_{35},  x_{25}x_{34}, x_{12}+x_{13},
x_{13}+x_{23}, & \\
 x_{13}-x_{34}+x_{35}, x_{12}x_{34}+x_{14}x_{23}, x_{12}x_{35}-x_{13}x_{25}) &.
\end{align*}
This last ideal is zero dimensional; this computation was done in Sage \cite{sagemath} and by hand.
Therefore, the intersection of the subspace $\PP(\ker \tilde{f})$ with the Grassmannian $\Gr(2,H^1(M(\A)))$ is (scheme theoretically) the union of five points. 
Since we have exhibit five distinct rational points, we obtain that the first resonance variety $\R^1(\A;\Q)$ is the union of the five corresponding planes.
\end{proof}

The map $T \to T$ defined by $(t_1,t_2) \mapsto (t_1, t_1^a t_2^n)$ is a cyclic Galois covering.
For every $n$ and $a$ the above map restricts to a Galois covering $\pi_a \colon M(\A_n^a) \to M(\A)$ with Galois group $\Z/n\Z$.
The map $\pi_a$ induces an inclusion
\[\pi_a^*: H^\bigcdot (M(\A);\Z) \hookrightarrow H^\bigcdot (M(\A_n^a);\Z)\]
of cohomology rings with integer coefficients. 

Since $n$ is coprime with $2$ and $3$, $H^1(M(\A_n^a);\Z)$ has rank five, equal to that of $H^1(M(\A);\Z)$.
Let $\alpha= \frac{1}{2 \pi \sqrt{-1}} \dd \log t_1$ and $\beta= \frac{1}{2 \pi \sqrt{-1}} \dd \log t_2$ be the two canonical generators of $H^1(T;\Z)$ as sub-lattice of $H^1(M(\A_n^a);\Z)$: then the morphism $\pi^*_a$ is
\begin{align*}
\pi_a^*(\psi_1)&= \alpha, \\
\pi_a^*(\psi_2)&= n \beta + a \alpha, \\
\pi_a^*(\omega_i) & =\omega_i \qquad \textnormal{for }i=1,2,3.
\end{align*}

\begin{lemma}
The first resonance variety $\R^1(\A_n^a;\Z)$ is the union of the following five sub-lattices of $H^1(M(\A_n^a);\Z)$:
\begin{align*}
Q_1 &=\langle \omega_1, \alpha\rangle, \\
Q_2 &=\langle \omega_2, n\beta + a\alpha \rangle, \\
Q_3 &=\langle \omega_3, n\beta + (a+1)\alpha \rangle, \\
Q_4 &=\langle \omega_1-\omega_3, \omega_2-\omega_1 + \alpha \rangle, \\
Q_5 &=\langle \omega_2-\omega_3, \omega_1-\omega_2 + n\beta + a\alpha \rangle.
\end{align*}
\end{lemma}
\begin{proof}
For $i=1,2$, the lattice $H^i(M(\A_n^a);\Z)$ is embedded in $H^i(M(\A);\Q)$ and the first resonance variety $R^1(\A_n^a;\Z)$ is the intersection
\[\R^1(\A_n^a;\Z)=\R^1(\A;\Q) \bigcap H^1(M(\A_n^a);\Z). \qedhere\]
\end{proof}

Now we can complete the proof of \Cref{thm:main_integer}.

\begin{proof}[Proof of \Cref{thm:main_integer}]
The posets of layers $\S(\A_n^1)$ and $\S(\A_n^2)$ are isomorphic because they have $3$ connected hypertori that intersect in $n$ points $(1, \zeta_n^i)$ for $i=0, \dots,n-1$ (where $\zeta_n$ is a $n^{\textnormal{th}}$ primitive root of unity).
The Hasse diagram of the posets of layers in the case $n=7$ is represented in \Cref{fig:Hasse_diag}.
\begin{figure}
\centering
\begin{tikzpicture}
\draw \foreach \x in {1,2,...,7}
    {(\x-1,2) node [circle, draw, fill=black, inner sep=0pt, minimum width=4pt,label={[above]$p_\x$}] {}
    \foreach \y in {0,1,2}
    {
      (\x-1,2)   -- (2*\y+1,1)}
    };
\draw (3,0) node [circle, draw, fill=black, inner sep=0pt, minimum width=4pt] {} -- (1,1) node [circle, draw, fill=black, inner sep=0pt, minimum width=4pt] {};
\draw (3,0)-- (3,1) node [circle, draw, fill=black, inner sep=0pt, minimum width=4pt] {};
\draw (3,0)-- (5,1) node [circle, draw, fill=black, inner sep=0pt, minimum width=4pt] {};
\draw (3,0) node [below]{$T$};
\draw (1,1) node [left]{$H_1$};
\draw (3,1) node [below]{$H_2$};
\draw (5,1) node [right]{$H_3$};
\end{tikzpicture}
\caption{The Hasse diagram of the poset of layers of $\A_7^1$ which coincides with the one of $\A_7^2$}\label{fig:Hasse_diag}
\end{figure}
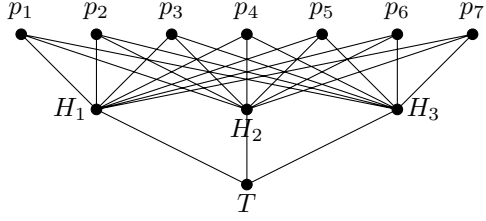
Suppose that there exists an isomorphism $\varphi: H^\bigcdot(M(\A^1_n);\Z) \rightarrow H^\bigcdot(M(\A^2_n);\Z)$; then $\varphi$ must map $\R^1(\A_n^1;\Z)$ isomorphically into $\R^1(\A_n^2;\Z)$.
Furthermore, $\varphi$ sends each component $Q_i^1$ into another component $Q_{f(i)}^2$.
For each $(i,j)$, consider the cardinality $c^a (i,j)$ of the torsion subgroup of $H^1(M(\A^a_n);\Z)/\langle Q_i^a,Q_j^a\rangle$ for $a=1,2$.
The value of $c^a(i,j)$ is $n$ when $(i,j)=(1,2),(1,3),(2,3),(4,5)$ and $1$ otherwise, both for $a=1$ and $a=2$.
Thus, $\varphi$ maps $Q_1^1,Q_2^1,Q_3^1$ into $Q_1^2,Q_2^2,Q_3^2$ in some order.
We define, for any sub-lattice $\Lambda$ of $H^1(M(\A_n^a);\Z)$, its radical $\Rad \Lambda$:
\[ \Rad \Lambda = \{x \in H^1(M(\A_n^a);\Z) \mid \exists \, n \in \N_+ \textnormal{ such that } nx \in \Lambda  \} .\]
For $a=1,2$ the following equality holds
\[ H^1((\C^*)^2;\Z)= \Rad \left( \bigcap_{1\leq i < j \leq 3} \langle Q_i^a,Q_j^a \rangle \right),\]
hence $\varphi$ preserves the sub-lattice $L \defeq H^1((\C^*)^2;\Z)=\langle \alpha, \beta \rangle$.
Now we claim that there is no linear map $\varphi_{|L}: L \rightarrow L$ that sends the three sub-lattices 
\[\{ Q_1^1 \cap L, Q_2^1 \cap L, Q_3^1 \cap L\} \]
into $\{ Q_1^2 \cap L, Q_2^2 \cap L, Q_3^2 \cap L\} $ in some order.
The three one-dimensional lattices are 
$Q_1^1 \cap L=\langle \alpha \rangle$, $Q_2^1 \cap L=\langle n\beta+\alpha \rangle$, $Q_3^1 \cap L= \langle n\beta+2\alpha \rangle$
for the arrangement $\A_n^1$ and the lattices
$ Q_1^2 \cap L=\langle \alpha \rangle$, $ Q_2^2 \cap L=\langle n\beta+2\alpha \rangle$, $ Q_3^2 \cap L=\langle n\beta+3\alpha \rangle$
for the arrangement $\A_n^2$.
In the case $a=1$ we can find generators for two of those lattices (e.g. $-\alpha$ and $n\beta+\alpha$) such that their sum belongs to the sub-lattice $nL$.
This property does not hold for the arrangement $\A_n^2$: indeed $\pm \alpha \pm (n\beta+2\alpha), \pm \alpha \pm (n\beta+3\alpha), \pm (n\beta+2\alpha) \pm (n\beta+3\alpha)$ are not in $nL$ (here we use $n\neq 5$).
Thus, we conclude that the map $\varphi$ cannot exist.
\end{proof}

\section{Second example}\label{sect:second_example}
The following example is constructed by looking for two toric arrangements with the following properties.
The underline matroid is not a modular matroid.
The two toric arrangements are coverings of the same toric arrangement with non cyclic Galois group.
The smallest example of such arrangements must have rank at least three and four hypertori.

Consider the three matrices
\[ N=\begin{pmatrix}
1 & 1 & 1 & 3 \\ 
0 & 5 & 0 & 5 \\ 
0 & 0 & 5 & 5
\end{pmatrix},   \quad
N'=\begin{pmatrix}
1 & 4 & 1 & 6 \\ 
0 & 5 & 0 & 5 \\ 
0 & 0 & 5 & 5
\end{pmatrix},  \quad
 N''=\begin{pmatrix}
1 & 0 & 0 & 1 \\ 
0 & 1 & 0 & 1 \\ 
0 & 0 & 1 & 1
\end{pmatrix}.\]
These integer matrices describe three central toric arrangements $\A$, $\A'$ and $\A''$ in $T=(\C^*)^3$.
Both $\A$ and $\A'$ are Galois coverings of $\A''$ with Galois group $\Z/5\Z \times \mathbb{Z}/5\mathbb{Z}$.

Let $([4],\rk,m)$ be the arithmetic matroid defined by $\rk (S)= \min (|S|,3)$ and by
\[ m(S)=\begin{cases} 1 & \textnormal{if }|S|\leq 1 \\
5 & \textnormal{if }|S|=2 \\
25 & \textnormal{if }|S|\geq 3 
\end{cases} .\]
Let $\mathcal{M}$ be the matroid over $\Z$ defined by
\[ \mathcal{M}(S)=\begin{cases} 
\Z^3 & \textnormal{if }|S|= 0 \\
\Z^2 & \textnormal{if }|S|= 1 \\
\Z \times \Z/5\Z & \textnormal{if }|S|=2 \\
\Z/5\Z  \times \Z/5\Z  & \textnormal{if }|S|\geq 3 
\end{cases} .\]
\begin{thm}\label{thm:main_combin}
The matrices $N$ and $N'$ are representations of the arithmetic matroid $([4],\rk,m)$ and of the matroid $\mathcal{M}$ over $\Z$.
Moreover, the posets $\S(\A)$ and $\S(\A')$ are not isomorphic.
\end{thm}
\begin{proof}
The first assertion follows from the Smith normal form of $N[S]$ and of $N'[S]$, the matrices obtained from $N$ and $N'$ by taking only the columns indexed by $S$.
The second one follows from \Cref{lemma:not_iso} below.
\end{proof}

The Poincaré polynomials of the complements $M(\A)$ and $M(\A')$ coincide with
\[P(t)=P'(t)=110 t^3+41 t^2+7 t+1.\]
The one of $M(\A'')$ is $P''(t)=14 t^3+17 t^2+7 t+1$.
The Tutte polynomial of the arithmetic matroid $([4],\rk,m)$ is $x^3+x^2+25x+25y+48$ and the one associated with $N''$ is $x^3+x^2+x+y$.

Define $a \vee b$ as the set of all least upper bound of $a,b$ in the poset of layers.
Consider the following property
\begin{equation}
\exists \, \{i,j\}\cup \{k,l\} = [4] 
\, \forall \, a \in i \vee j,\, \forall \, b \in  k \vee l \, ( a\vee b \neq \emptyset) \tag{P}. \label{eq:prop_poset}
\end{equation}
In other words, the property \eqref{eq:prop_poset} for $\S(\A)$ (or for $\S(\A')$) says that there exists a choice of two hypertori $H_i,H_j$ in $\A$ (resp. in $\A'$) such that every connected component of $H_i \cap H_j$ intersects every connected component of $H_k \cap H_l$.
\begin{lemma}\label{lemma:not_iso}
The property \eqref{eq:prop_poset} holds for $\S(\A)$ but not for $\S(\A')$.
\end{lemma}

\begin{proof}
We first discuss the poset $\S(\A')$.
Consider $(i,j,k,l)=(1,2,3,4)$, there are five possible joins $1\vee 2$ that correspond to the five layers 
\[ a_\mu :  \begin{cases}
x=1 \\ 
y=\mu \\ 
\end{cases} ,\]
where $\mu$ runs over all the fifth roots of unity.
Analogously, the joins of $3$ and $4$ are the five layers
\[ b_\zeta : \begin{cases}
x=z^{-5} \\ 
y=\zeta z^{5} \\ 
\end{cases} ,\]
where $\zeta$ runs over all the fifth roots of unity.
A join  $a_\mu \vee b_\zeta$ exists if and only if the system
\begin{equation}
\label{eq:system}
 \begin{cases}
x=1 \\ 
y=\mu \\
z^5=1 \\ 
y=\zeta z^{5}
\end{cases}. 
\end{equation}
admits a solution.
If $\zeta= \mu$, then the system has five solutions, otherwise there are no solutions.
In particular, the property \eqref{eq:prop_poset} does not hold for the poset $\S(\A')$.

The following case by case analysis shows that the three systems for the arrangement $\A$ analogous to \eqref{eq:system} have always a unique solution:
\[\begin{cases}
x=1 \\ 
y=\mu \\ 
xz^5=1 \\
xyz^{3}=\zeta  
\end{cases},\quad
\begin{cases}
x=1 \\ 
z=\mu \\
xy^5=1 \\ 
x^2y^3z=\zeta
\end{cases},\quad
\begin{cases}
x=1 \\ 
yz=\mu \\
xy^5=1 \\ 
y=\zeta z
\end{cases}.\qedhere
 \] 
\end{proof}

\begin{prop}
The spaces $M(\A)$ and $M(\A')$ have non-isomorphic cohomology algebras with rational coefficients, i.e.\
\[ H^\bigcdot(M(\A);\Q) \not \simeq H^\bigcdot(M(\A');\Q).\]
\end{prop}

\begin{proof}
Suppose that an isomorphism $\varphi: H^\bigcdot(M(\A);\Q)  \rightarrow H^\bigcdot(M(\A');\Q)$ exists.
We claim that $\varphi(H^\bigcdot(T;\Q))=H^\bigcdot(T;\Q)$ where $T$ is the ambient torus.
The proof of the claim is analogous to the one of \Cref{lemma:R1(A)}.
The first resonance variety of $M(\A)$ and $M(\A')$ are the union of the four planes
\begin{align*}
Q_1 &=\langle \omega_1, \alpha\rangle, & 
Q_1' &=\langle \omega_1, \alpha\rangle, \\
Q_2 &=\langle \omega_2, 4\alpha + 5\beta \rangle, &
Q_2' &=\langle \omega_2, \alpha + 5\beta \rangle, \\
Q_3 &=\langle \omega_3, \alpha + 5\gamma \rangle, &
Q_3' &=\langle \omega_3, \alpha + 5\gamma \rangle, \\
Q_4 &=\langle \omega_4, 3\alpha + 5\beta + 5\gamma \rangle, &
Q_4' &=\langle \omega_4, 6\alpha + 5\beta + 5\gamma \rangle,
\end{align*}
since the unique relations in cohomology of degree two are $\omega_i \psi_i=0$ (see \Cref{thm6.13}).
Thus there exists a bijection $f:[4] \rightarrow [4] $ such that $\varphi$ sends $Q_i$ into $Q_{f(i)}'$, for $i=1,\dots,4$.
Since $H^1(T;\Q)= \bigcap_{i=1}^4 \langle Q_{j} \rangle_{j\neq i}$ in  $H^1(M(\A);\Q)$  and $H^1(T;\Q)= \bigcap_{i=1}^4 \langle Q_{j}' \rangle_{j\neq i}$ in  $H^1(M(\A');\Q)$, the map $\varphi$ preserves the subspace $H^\bigcdot(T;\Q)$.
Consider now the quotients $S^{\bigcdot}=H^\bigcdot(M(\A);\Q)/(H^1(T;\Q))$ and $S^{\prime \bigcdot}=H^\bigcdot(M(\A');\Q)/(H^1(T;\Q))$.
The multiplication map $S^1 \times S^2 \rightarrow S^3$ has rank $51$, instead the map $S^{\prime 1} \times S^{\prime 2} \rightarrow S^{\prime 3}$ has rank $43$.
The rank of the two multiplication maps can be calculated with a computer.
Therefore the map $\varphi$ cannot be an isomorphism.
\end{proof}

The difference between the rank of $S^1 \times S^2 \rightarrow S^3$ and $S^{\prime 1} \times S^{\prime 2} \rightarrow S^{\prime 3}$ can be explained intuitively.

For every $I \subseteq [4]$, the set of connected components of $\cap_{i \in I} H_i$ has a natural group structure, induced by the ambient torus.
We call this group $\LG(I)$.
Moreover given $I\subset J \subseteq [4]$, there exists a natural group homomorphism $\pi \colon \LG(J) \to \LG(I)$ that maps a connected component $W$ to the unique connected component of $\bigcap_{i \in I} H_i$ containing $W$.
In our case, since $\bigcap_{j\neq i} H_j= \bigcap_{j \in [4]} H_j$ for all $i \in [4]$, the map $\LG([4]) \to \LG([4] \setminus \{i\})$ is the identity.
Call $\pi_{i,j} \colon \LG([4]) \to \LG(\{i,j\})$ the canonical projection.

Given $I$ and $J$ of cardinality two, there exists an isomorphism $\varphi_I^J$ such that the following diagram commutes
\begin{center}
\begin{tikzcd}
\LG([4]) \arrow[r, "\pi_I"] \arrow[d,"\pi_J"] & \LG(I) \arrow[ld, "\varphi_I^J", dashed] \\
\LG(J) 
\end{tikzcd}
\end{center}
if and only if $\ker \pi_I=\ker \pi_J$.
We compute all these kernels both for $\A$ and $\A'$ and we report it in \Cref{table:kernels}, where $e_1$ and $e_2$ are two generators of $\LG([4])\simeq \Z/5\Z \times \Z/5\Z$.

\begin{table}
\caption{For every set $I\subset [4]$, $|I|=2$, we describe the kernel of $\pi_I$ and of $\pi'_I$.}
\centering
\label{table:kernels}
\begin{tabular}{c|c|c}
$I$ & $ \ker \pi_I$ & $\ker \pi'_I$ \\ 
\hline 
$\{1,2\}$ & $\langle e_2 \rangle$ & $\langle e_2 \rangle$ \\ 
$\{1,3\}$ & $\langle e_1 \rangle$ & $\langle e_1 \rangle$ \\ 
$\{1,4\}$ & $\langle e_1-e_2 \rangle$ & $\langle e_1-e_2 \rangle$ \\ 
$\{2,3\}$ & $\langle 4e_1-e_2 \rangle$ & $\langle e_1-e_2 \rangle$ \\
$\{2,4\}$ & $\langle 2e_1-e_2 \rangle$ & $\langle 3e_1-e_2 \rangle$ \\ 
$\{3,4\}$ & $\langle 3e_1-e_2 \rangle$ & $\langle e_2 \rangle$ \\ 
\end{tabular}
\end{table}

From \Cref{thm6.13}, we have that $S^\bigcdot$ is generated by the image of $\oomega_i:= \oomega_{H_i,\{i\}}$ for $i\in [4]$, of $\oomega_{a,I}$ for $|I|=2$ and $a \in \LG(I)$, and of $\oomega_{b, [4] \setminus \{i\}}$ for $i\in [4]$ and $b\in \LG([4])$.
The linear relations are 
\[ \sum_{i=1}^4 (-1)^i \oomega_{b,[4]\setminus \{i\}}=0\]
for each $b\in \LG([4])$.
The product $S^1 \times S^2 \to S^3$ is defined by
\[ \oomega_i \oomega_{a,\{j,k\}} = (-1)^{l(\{i\},\{j,k\})} \sum_{b \in \pi^{-1}_{j,k}(a)} \oomega_{b,\{i,j,k\}}.\]
The analogous definitions and formulas hold for the arrangement $\A'$.
In the algebra $S^{\prime \bigcdot}$ the following relations hold for $a \in \LG'(\{1,2\})$ and $c \in \LG'(\{1,4\})$:
\begin{align*}
&(\oomega_1'-\oomega_2'+ \oomega_3' - \oomega_4')(\oomega_{a,\{1,2\}}' +\oomega_{\varphi_{1,2}^{3,4}(a),\{3,4\}}')=0, \\
&(\oomega_1' + \oomega_2' - \oomega_3' - \oomega_4')(\oomega_{c,\{1,4\}}' + \oomega_{\varphi_{1,4}^{2,3}(c),\{2,3\}}')=0,
\end{align*}
since $\ker \pi'_{1,2}= \ker \pi'_{3,4}$ and $\ker \pi'_{1,4} = \ker \pi'_{2,3}$.
These equations give ten independent relations, the corresponding relations in the algebra $S^\bigcdot$ are only two:
\begin{align*}
&(\oomega_1 - \oomega_2+ \oomega_3 - \oomega_4) \bigg( \sum_{ a\in \LG(\{1,2\})} \oomega_{a,\{1,2\}} + \sum_{b \in \LG(\{3,4\})} \oomega_{b, \{3,4\}} \bigg) =0, \\
&(\oomega_1 + \oomega_2 - \oomega_3 - \oomega_4) \bigg( \sum_{c\in \LG(\{1,4\})} \oomega_{c,\{1,4\}} + \sum_{ d \in \LG(\{2,3\})} \oomega_{d,\{2,3\}} \bigg) =0,
\end{align*}
since $\ker \pi_{1,2} \neq \ker \pi_{3,4}$ and $\ker \pi_{1,4} \neq \ker \pi_{2,3}$.

By \cite[Theorem E]{Delucchi-Riedel}, the $G$-semimatroids described by $N$ and $N'$ are different.

\printbibliography

\bigskip\bigskip
\end{document}